\documentclass[12pt,reqno]{amsart}

\usepackage{amsmath,amssymb,amsthm,pifont,bbm,color,graphicx,hyperref}
\usepackage[numbers,sort&compress]{natbib}
\usepackage[T1]{fontenc}
\usepackage[
  margin=1.2in,
  marginparwidth=7mm,
  marginparsep=1mm,
]{geometry}

\newcommand\nc\newcommand
\nc\Z{\mathbb Z}
\nc\E{\mathbb E}
\nc\eqd{\overset{d}{=}}
\nc\Star{\raisebox{-1pt}{\mbox{\ding{65}}}}

\nc\dc\DeclareMathOperator
\dc\Cycle{\mu}

\nc\rc\renewcommand
\rc\P{\mathbb P}
\rc\tilde\widetilde
\rc\hat\widehat
\rc\labelenumi{(\roman{enumi})}
\rc\theenumi\labelenumi

\nc\nt\newtheorem
\nt{thm}{Theorem}
\nt{prop}[thm]{Proposition}
\nt{lemma}[thm]{Lemma}
\nt{cor}[thm]{Corollary}
\nt{alg}[thm]{Algorithm}

\title{Finitely dependent cycle coloring}

\date{28 July 2017}

\author[Alexander E. Holroyd]{Alexander E.\ Holroyd}
\address{Alexander E.\ Holroyd} \email{holroyd@math.ubc.ca}
\urladdr{\url{http://aeholroyd.org}}

\author{Tom Hutchcroft}
\address{Tom Hutchcroft, Department of Mathematics, University of British Columbia}
\email{tomhutchcroft@gmail.com}

\author{Avi Levy}
\address{Avi Levy, Department of Mathematics, University of Washington}
\email{avius@math.washington.edu}
\urladdr{\url{http://math.washington.edu/~avius}}


\begin{document}

\begin{abstract}
  We construct stationary finitely dependent colorings of the cycle which are analogous to the colorings of the integers recently constructed by Holroyd and Liggett. These colorings can be described by a simple necklace insertion procedure, and also in terms of an Eden growth model on a tree. Using these descriptions we obtain simpler and more direct proofs of the characterizations of the 1- and 2-color marginals.
\end{abstract}

\maketitle

\section{Introduction}
  A random process indexed by the vertex set of a graph is $k$\textbf{-dependent} if its restrictions to any two sets of vertices at graph distance greater than $k$ are independent of each other. A process is \textbf{finitely dependent} if it is $k$-dependent for some finite $k$. For several decades it was not known whether every stationary finitely dependent process on $\Z$ is a \textbf{block factor} of an i.i.d.\ process, that is, the image of an i.i.d.\ process under a finite range function that commutes with translations. This question was raised by Ibragimov and Linnik \cite{ibragimov1965independent} in 1965 and resolved in the negative by Burton, Goulet, and Meester \cite{burton1993} in 1993.  Recently Holroyd and Liggett \cite{HL} proved that \emph{proper coloring} distinguishes between block factors and finitely dependent processes: block factor proper colorings of $\Z$ do not exist, but finitely dependent stationary proper colorings do.  In fact, these colorings fit into a more general family constructed subsequently in \cite{hhlMalCol}. See also \cite{H,HL2}.

  The finitely dependent colorings of \cite{HL} have short but mysterious descriptions. An interesting feature of these colorings is that the supports of proper subsets of the colors can each be expressed as block factors of i.i.d.\ \cite[Theorem~4]{HL}, even though the coloring as a whole cannot \cite[Proposition~2]{HL}.  Moreover, their descriptions as block factors are remarkably simple and explicit.  However, the proofs used to obtain these descriptions in \cite{HL} were in some cases quite involved.
  Here we introduce a more canonical construction via colorings of the $n$-cycle, which can be expressed in terms of a necklace insertion process akin to those in \cite{mallowsSheppNecklace,nakataNecklace}, or in terms of the classical Eden growth model on a 3-regular tree. Using this new construction, we are able to obtain simpler and more direct proofs of the statements concerning subsets of colors mentioned above.

  The necklace insertion process is as follows. Suppose we have a necklace of colored beads. Start with 3 beads with uniformly random distinct colors from $\{1,\ldots,q\}$. At each step, pick a uniformly random gap between consecutive beads and insert a bead with uniformly random color differing from those of the two neighbors. After $n-1$ steps we have a coloring $(C_1,\ldots,C_{n+2})$ of the $(n+2)$-cycle.

  Here is a different description of the above process, which is easily seen to be equivalent (see Section~\ref{sec:index} for details).  Consider a planar embedding of the 3-regular tree $\mathbb T$ (with a distinguished root vertex), together with its planar dual $\mathbb D$, which is an infinite-degree triangulation (see Figure~\ref{fig:cycleColor}).  For a vertex $v$ of $\mathbb T$, let $\Delta(v)$ be the set of three vertices in $\mathbb D$ incident to the face dual to $v$.  We consider the Eden growth model \cite{eden1961two} on $\mathbb T$, which is a random growing sequence of clusters $T_1,T_2,\ldots$ defined as follows.  The initial cluster $T_1$ consists of the root vertex.  Given $T_n$, we choose a vertex uniformly at random from those adjacent to but not lying in $T_n$, and add it to $T_n$ to form the new cluster $T_{n+1}$.  Let $D_n$ be the subgraph of $\mathbb{D}$ induced by the vertex set $\bigcup_{v \in T_n} \Delta(v)$. The graph $D_n$ inherits a planar embedding from $\mathbb D$, in which there is one outer face containing all of its $n+2$ vertices, and all other faces are triangles. Let $q\geq 3$ be an integer.  Conditional on $D_n$, choose a uniformly random proper $q$-coloring of $D_n$ and, independently, a uniform vertex $u$ of $D_n$. Let $C_1,\ldots,C_{n+2}$ be the sequence of colors of the vertices of the outer face in clockwise order starting from $u$.

  \begin{figure}[t]
    \includegraphics[width=0.49\textwidth]{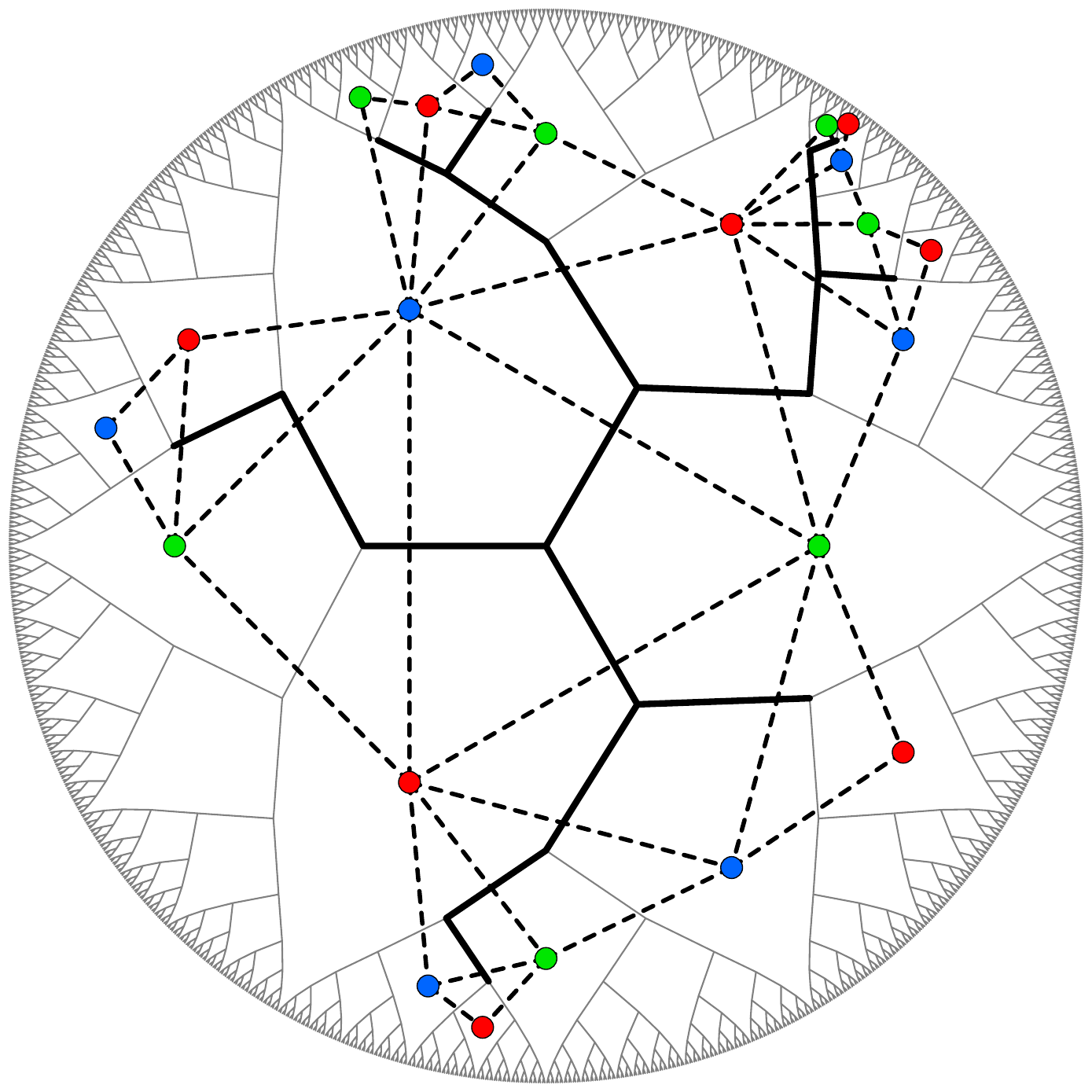}
    \hfill
    \includegraphics[width=0.49\textwidth]{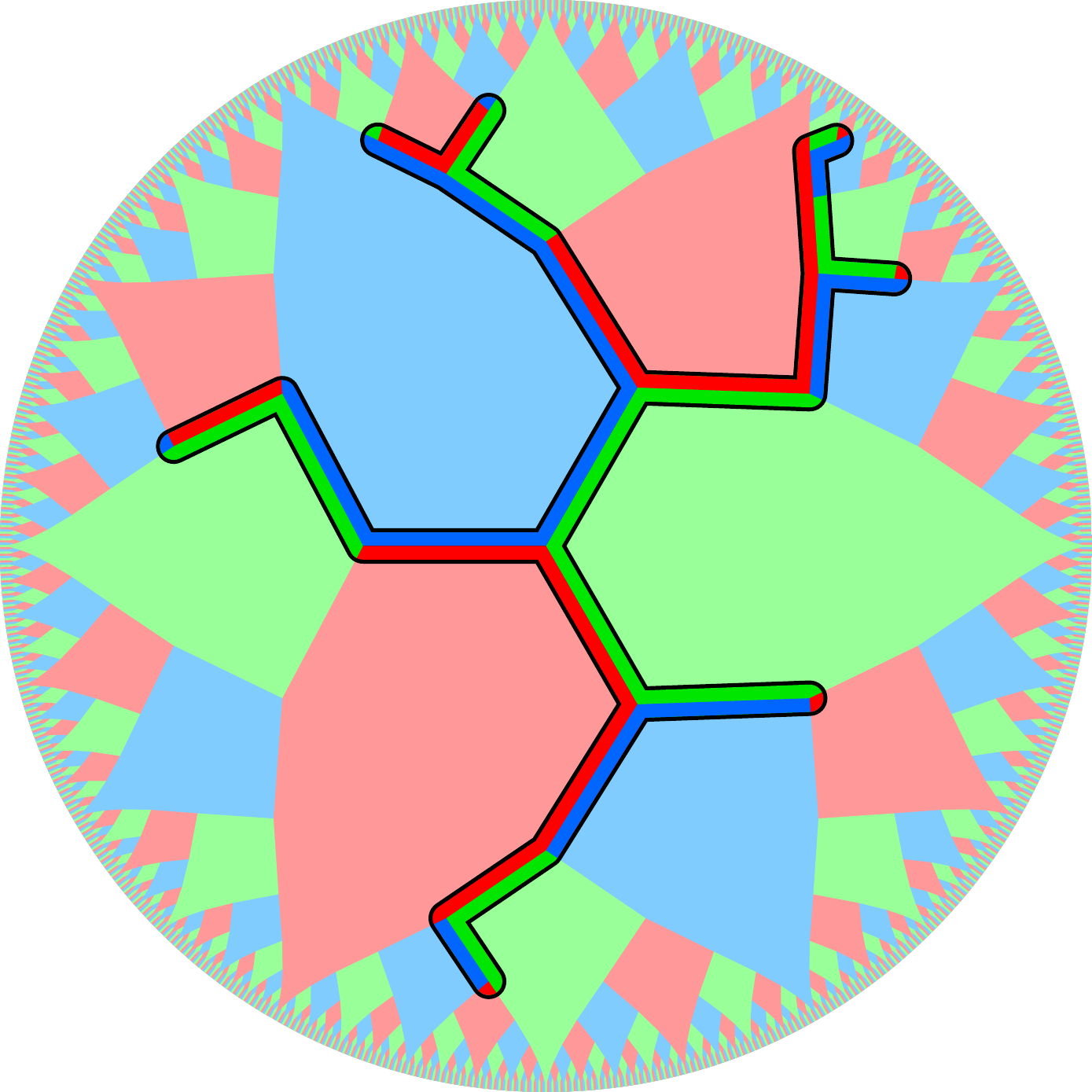}
    \caption{\label{fig:cycleColor} Two versions of the construction of the 2-dependent 3-coloring of the cycle, from a cluster of the Eden model on a 3-regular tree. On the left we uniformly properly 3-color the planar map comprising the dual triangles of the vertices of the cluster. The coloring is read clockwise around the outer face.  On the right, we may alternatively fix a uniform proper coloring of the infinite dual map in advance. The coloring is read around the outer boudary of the cluster.}
  \end{figure}

  \begin{samepage}
  \begin{thm}\label{thm:main}
    Fix $n\geq 1$ and $(k,q)\in \{(1,4),(2,3)\}$. The sequence $(C_1,\ldots,C_{n+2})$ constructed above is a $k$-dependent proper $q$-coloring of the $(n+2)$-cycle.  The coloring is symmetric in law under rotations and reflections of the cycle and permutations of the colors.  Moreover, the sequence $(C_1,\ldots,C_{n+2-k})$ is equal in law to $(X_1,\ldots,X_{n+2-k})$ where $(X_i)_{i\in\Z}$ is the stationary $k$-dependent $q$-coloring of $\Z$ constructed in \cite{HL}.
  \end{thm}
  \end{samepage}

  A third description of our construction is as follows. It is possible to define the uniformly random proper $q$-coloring of the infinite graph $\mathbb D$ (by consistency). When $q=3$ this coloring is simply a uniform choice from the $3!$ proper $q$-colorings of $\mathbb D$, corresponding to the permutations of the colors. (On the other hand when $q=4$, there are uncountably many such colorings.) We can now choose the Eden model cluster $T_n$ independently of the coloring of $\mathbb D$. The coloring of Theorem~\ref{thm:main} then arises as the sequence of colors appearing in clockwise order around the outer boundary of $T_n$. See Figure~\ref{fig:cycleColor}.

  We next address the one- and two-color marginals of the colorings.

  \begin{thm}\label{thm:onetwo}
    Let $I$ be either $\Z$, or $\{1,\ldots,n\}$ for $n\geq 3$ (interpreted as the vertex set of a cycle). Let $X=(X_i)_{i\in I}$ be the 1-dependent 4-coloring and let $Y=(Y_i)_{i\in I}$ be the 2-dependent 3-coloring of $I$ (arising from \cite{HL} in the $\Z$ case, or from Theorem~\ref{thm:main} in the cycle case).
    \begin{enumerate}
      \item The process $(\mathbbm{1}[X_i\in\{1,2\}])_{i\in I}$ is equal in law to $(\mathbbm{1}[U_i>U_{i+1}])_{i\in I}$, where $(U_i)_{i\in I}$ are i.i.d.\ uniform on $[0,1]$.
      \item The process $(\mathbbm{1}[Y_i=1])_{i\in I}$ is equal in law to $(\mathbbm{1}[U_{i-1}<U_i>U_{i+1}])_{i\in I}$, where $(U_i)_{i\in I}$ are i.i.d.\ uniform on $[0,1]$.
      \item The process $(\mathbbm{1}[X_i=1])_{i\in I}$ is equal in law to $(\mathbbm{1}[B_i>B_{i+1}])_{i\in I}$, where $(B_i)_{i\in I}$ are i.i.d.\ taking values $0,1$ with equal probabilities.
    \end{enumerate}
    All indices are interpreted modulo $n$ in the cycle case.
  \end{thm}

  Color symmetry (Theorem~\ref{thm:main}) of course implies that any other colors may be substituted for $1$ and $2$ in (i)--(iii) above.
  The $\Z$ case of this theorem was proved in \cite[Theorem~4]{HL}, but the proof was quite complicated and mysterious. We will deduce the $\Z$ case from the cycle case, which in turn will be proved by very simple and direct methods.  The idea will be to construct the various cycle-indexed processes via Markovian necklace insertion procedures like the one already mentioned, and to couple them to each other.

  The insertion procedure corresponding to part (ii) was studied earlier in \cite{mallowsSheppNecklace,nakataNecklace}.  The fact that this procedure results in the peak set of a random permutation of the cycle (as follows from our proof of (ii) above) gives simpler proofs of many of the results of those papers.

  Notwithstanding the simplicity of our proof of Theorem~\ref{thm:onetwo}, these characterizations of marginal processes remain rather
  mysterious.  In particular, we do not know whether a similar characterization exists for the following variant.  Consider the function $\iota:\{1,2,3,4\}\to\{1,2,\Star\}$ given by
  $$\iota(1)=1;\quad \iota(2)=2; \quad \iota(3)=\iota(4)=\Star.$$
  Letting $(X_i)_{i\in\Z}$ be the $1$-dependent $4$-coloring of $\Z$ from
  \cite{HL}, we do not know whether the process $(\iota(X_i))_{i\in \Z}$ has a simple description similar to those in Theorem~\ref{thm:onetwo}.  More precisely, a process
  $(Z_i)_{i\in\Z}$ is called $r$-block-factor of an i.i.d.\ process $(U_i)_{i\in\Z}$ if
  $Z_i=f(U_{i},\ldots,U_{i+r-1})$ for all $i$ and a fixed function $f$. Every $(k+1)$-block-factor of an i.i.d.\ process is $k$-dependent (but no coloring is a block factor of an i.i.d.\ process \cite[Proposition 2]{HL}).  We do not know whether $(\iota(X_i))_{i\in \Z}$ is a block factor of an i.i.d.\ process.  Since the process is $1$-dependent, one might expect it to be a 2-block factor (by analogy with Theorem~\ref{thm:onetwo}).  We show that this is not the case.

  \begin{thm}\label{onetwostar}
  Let $(X_i)_{i\in\Z}$ be the $1$-dependent $4$-coloring of $\Z$ from
  \cite{HL}.  The process $(\iota(X_i))_{i\in\Z}$ cannot be expressed as a
  $2$-block-factor of an i.i.d.\ process.
  \end{thm}

  In Section~\ref{sec:rec} we give an alternative combinatorial construction of the colorings in terms of recurrences. We prove Theorem~\ref{thm:main} in Section~\ref{sec:index} after showing that the coloring arising from the Eden model coincides with the combinatorial construction. Theorem~\ref{thm:onetwo} is proved in Section~\ref{sec:markov} and Theorem~\ref{onetwostar} in Section~\ref{sec:block}.

\section{Recurrences}\label{sec:rec}
  For a positive integer $m$, we write $[m]:=\{1,\ldots,m\}$. Fix a positive integer $q$, representing the number of colors. For an integer $n\geq 0$, a \textbf{word} $x$ of length $n$ is an element of $[q]^n$, which we write as $x=x_1x_2\cdots x_n$. For words $x$ and $y$ we write $xy$ for their concatenation. The word $x$ is \textbf{proper} if $x_i\not=x_{i+1}$ for all $i\in [n-1]$. The word $x$ is \textbf{cyclically proper} if it is proper and in addition $x_n\not=x_1$. For $i\in [n]$ and $x\in [q]^n$ we denote the word $x_1x_2\cdots x_{i-1}x_{i+1}\cdots x_n$ by $\hat x_i$. The empty word (i.e., the unique element of $[q]^0$) is denoted by $\emptyset$.

  Recursively define the function $B^{\circ}$ on $[q]^n$ via
  \begin{equation}\label{eq:rec}
    B^{\circ}(x)=\mathbbm{1}[\text{$x$ is cyclically proper}]\sum_{i=1}^nB^{\circ}(\hat x_i),\qquad n\geq 1,\ x\in [q]^n,
  \end{equation}
  and $B^{\circ}(\emptyset)=1$. Let $Z^\circ(n,q):=\sum_{y\in[q]^n}B^\circ(y)$. It is easily verified that $Z^\circ(n,q)>0$ for all $q\geq 3$. For all such integers $q$ and for all $n\geq 0$, let $\Cycle_{n,q}$ denote the probability measure on $[q]^n$ given by
  \begin{equation}\label{eq:cmarg}
    \Cycle_{n,q}(\{x\})=\frac{B^\circ(x)}{Z^\circ(n,q)}.
  \end{equation}

  \begin{lemma}\label{lem:shift}
    For $n\geq 0$ and $q\geq 3$, the measure $\Cycle_{n,q}$ is invariant under cyclic shifts.
  \end{lemma}
  \begin{proof}
    When $n=0$ there is nothing to prove.
    Fix $n\geq 1$ and $q$. We prove, by induction on $n$, that for all words $x\in[q]^n$ the cyclic shift $\hat x_1x_1$ of $x$ satisfies
    \begin{equation}\label{eq:shift}
      B^\circ(\hat x_1 x_1)=B^\circ(x).
    \end{equation}
    The base case $n=1$ is apparent. Now suppose that \eqref{eq:shift} holds for all words of length $n-1$ and let $x$ be a word of length $n$. Observe that $x$ is cyclically proper if and only if the word $y:=\hat x_1 x_1$ is cyclically proper. If neither $x$ nor $y$ has this property, then $B^\circ(x)=B^\circ(y)=0$ and \eqref{eq:shift} holds trivially. Otherwise,
    \begin{equation}\label{eq:xyrec}
      B^\circ(x)=\sum_{i=1}^nB^\circ(\hat x_i)\qquad\text{ and }\qquad B^\circ(y)=\sum_{i=1}^nB^\circ(\hat y_i).
    \end{equation}

    Observe that for all $i\in[n]$, the word $\hat y_i$ is a cyclic shift of $\hat x_{i+1}$ (where indices are interpreted modulo $n$). Thus $B^\circ(\hat y_i)=B^\circ(\hat x_{i+1})$ by the inductive hypothesis. It follows that the two sums in \eqref{eq:xyrec} are rearrangements of one another, and therefore $B^\circ(x)=B^\circ(y)$.
  \end{proof}

  The following is a useful reformulation of the recurrence \eqref{eq:rec}. It can be interpreted as an instance of the M\"obius inversion formula \cite[Section 3.7]{stanley1997enumerative}.
  \begin{lemma}\label{lem:mrec}
    For all $n,q\geq 1$ and for all words $x\in[q]^n$, we have that
    \begin{equation}\label{eq:mobrec}
      B^\circ(x)=\sum_{i=1}^nB^\circ(\hat x_i)-2\sum_{i=1}^{n}\mathbbm{1}[x_i=x_{i+1}]B^\circ(\hat x_i),
    \end{equation}
    where indices are interpreted modulo $n$.
  \end{lemma}
  \begin{proof}
    Fix the word $x$. If $x$ is cyclically proper, the latter term on the right side of \eqref{eq:mobrec} vanishes and the result reduces to the defining recurrence for $B^\circ$. If $x$ is not cyclically proper, the left side of \eqref{eq:mobrec} vanishes and it remains to verify that
    \begin{equation}\label{eq:smobrec}
      \sum_{i=1}^nB^\circ(\hat x_i)=2\sum_{i=1}^{n}\mathbbm{1}[x_i=x_{i+1}]B^\circ(\hat x_i).
    \end{equation}
    Fix an index $j\in[n]$ such that $x_j=x_{j+1}$. Then for all $i\in [n]\setminus \{j,j+1\}$, the word $\hat x_i$ is not cyclically proper. Hence the equality \eqref{eq:smobrec} we wish to verify is equivalent to
    $$
      B^\circ(\hat x_j)+B^\circ(\hat x_{j+1})=2B^\circ(\hat x_j),
    $$
    which follows since $\hat x_j=\hat x_{j+1}$.
  \end{proof}

  For $(k,q)\in\{(1,4),(2,3)\}$, the stationary $k$-dependent $q$-coloring $(X_i)_{i\in \Z}$ constructed in \cite{HL} is characterized as follows.
  Recursively define the function $\vec B$ on $[q]^n$ via
  \begin{equation}\label{eq:vrec}
    \vec B(x)=\mathbbm{1}[x\text{ is proper}]\sum_{i=1}^n \vec B(\hat x_i),\qquad n\geq 1,\ x\in [q]^n,
  \end{equation}
  and $\vec B(\emptyset)=1$. Let $\vec Z(n,q)=\sum_{y\in[q]^n}\vec B(y)$. Then for all $n\geq 0$ and for all $x\in [q]^n$,
  \begin{equation}\label{eq:vmarg}
    \mathbb P\bigl((X_1,\ldots,X_n)=x\bigr)=\frac{\vec B(x)}{\vec Z(n,q)}.
  \end{equation}

  In the following lemma and its proof, we use $\star$ to denote a dummy variable that is summed over $[q]$. For example, given a function $f\colon [q]^m\to \mathbb R$ and a word $x$, we write $f(x\star^m)$ as a shorthand for $\sum_{y\in [q]^m}$ $f(xy)$.

  \begin{lemma}\label{lem:rest}
  Fix integers $(k,q)\in \{(1,4),(2,3)\}$ and $n\geq 0$. Then for all words $x\in[q]^n$,
  \begin{equation}\label{eq:resteq}
      B^\circ(x\star^k)=Z^\circ(k,q)\vec B(x).
    \end{equation}
  \end{lemma}
  \begin{proof}
    We establish \eqref{eq:resteq} by induction on $n$. When $n=0$ or when $x$ is non-proper it holds trivially. Now suppose that \eqref{eq:resteq} holds for all words of length $n-1$ and that $x$ is proper. For any word $y\in [q]^k$, we have by Lemma~\ref{lem:mrec} that
    \begin{equation*}\label{eq:rbc}
      B^\circ(xy)=\sum_{i=1}^{n+k}B^\circ(\hat{xy}_i)-2\mathbbm{1}[x_n=y_1]B^\circ(\hat x_ny)-2\sum_{i=1}^k\mathbbm{1}[y_i=y_{i+1}]B^\circ(x\hat y_i),
    \end{equation*}
    where we have set $y_{k+1}=x_1$ (this occurs in the final summand of the second term). The term $2\mathbbm{1}[x_n=y_1]B^\circ(\hat x_ny)$ in the right side of the above equation is seen to equal $2\mathbbm{1}[x_n=y_1]B^\circ(x\hat y_1)$, and thus
    \begin{equation}\label{eq:rbc2}
      B^\circ(xy)=\sum_{i=1}^{n+k}B^\circ(\hat{xy}_i)-2\mathbbm{1}[x_n=y_1]B^\circ(x\hat y_1)-2\sum_{i=1}^k\mathbbm{1}[y_i=y_{i+1}]B^\circ(x\hat y_i).
    \end{equation}

    Next we sum \eqref{eq:rbc2} over $y\in [q]^k$ to obtain that
    \begin{align*}
      B^\circ(x\star^k)&=\sum_{i=1}^{n+k}\sum_{y\in[q]^k}B^\circ(\hat{xy}_i)-2(k+1)B^\circ(x\star^{k-1})\\
      &=\sum_{i=1}^nB^\circ(\hat x_i\star^k)+qkB^\circ(x\star^{k-1})-2(k+1)B^\circ(x\star^{k-1}),
    \end{align*}
    where the coefficient $qk$ in the second term counts the number of choices of which index of $y$ to delete together with the deleted value, and the coefficient $2(k+1)$ in the third term arises from the $k+1$ indicator functions appearing in \eqref{eq:rbc2}.
    Since $qk=2(k+1)$ for both pairs $(k,q)\in\{(1,4),(2,3)\}$, it therefore follows that
    $$
      B^\circ(x\star^k)=\sum_{i=1}^nB^\circ(\hat x_i\star^k).
    $$
    Substituting the inductive hypothesis into this equation, using \eqref{eq:vrec}, and recalling that $x$ is proper yields \eqref{eq:resteq}.
  \end{proof}

  \begin{lemma}\label{lem:partition}
    For all $n\geq 2$ and for all $q\geq 2$, we have that
    $$
      Z^\circ(n,q)=n!\, q(q-1)(q-2)^{n-2}.
    $$
  \end{lemma}
  \begin{proof}
    The proof is by induction on $n$. The base case $n=2$ is trivial to verify. Now suppose that $Z^\circ(n,q)=n!\,q(q-1)(q-2)^n$, where $n\geq 2$. Summing \eqref{eq:rec} over all words yields that
    \begin{equation}\label{eq:zc}
      Z^\circ(n+1,q)=\sum_{i=1}^{n+1}\sum_{x\in[q]^{n+1}}\mathbbm{1}[x\text{ is cyclically proper}]B^\circ(\hat x_i).
    \end{equation}
    For all $i\in[n+1]$, it is easy to see that
    \begin{equation}\label{eq:zc2}
      \sum_{x\in[q]^{n+1}}\mathbbm{1}[x\text{ is cyclically proper}]B^\circ(\hat x_i)=(q-2)Z^\circ(n,q).
    \end{equation}
    Combining \eqref{eq:zc}, \eqref{eq:zc2}, and the inductive hypothesis yields the result.
  \end{proof}

\section{Indexing}\label{sec:index}
  We introduce a general procedure for constructing cycle-indexed processes by random insertion. We will apply this first to the coloring itself; later we will apply it to other processes in order to prove Theorem~\ref{onetwostar}.
  To establish our remaining results we consider the following dynamical probabilistic insertion scheme. Let $n\geq 3$. Given a random sequence $X^{(n)}=(X_1,\ldots,X_n)$, we will choose a uniformly random index $I\in[n]$ (which will be independent of $X^{(n)}$) and insert a new element $Z$ (which can depend on $I$ and $X^{(n)}$) just before element $X_I$ to give
  \begin{equation}\label{eq:insert}
    (X_1,\ldots,X_{I-1},Z,X_I,\ldots,X_n).
  \end{equation}

  Here an indexing issue arises. We want to interpret the above sequence as a labelling of the $(n+1)$-cycle, but we would like the random insertion location to be uniform \textit{after} the insertion has taken place as well as before it. Simply defining $X^{(n+1)}$ to equal \eqref{eq:insert} does not achieve this (because the inserted element $Z$ can never end up at location $n+1$, for instance). Since all the random sequences that we consider are invariant under rotations of the cycle, we instead apply a random rotation after the insertion step. That is, we let $R$ be uniformly random in $[n+1]$ independent of all previous choices, and set $X^{(n+1)}$ to be the image of \eqref{eq:insert} under the map $\tau_R$, where $\tau_r$ is the rotation
  $$
    \tau_r\colon(x_1,\ldots,x_{n+1})\mapsto (x_{r+1},\ldots,x_{r+n+1}),
  $$
  where the indices are interpreted modulo $n+1$.

  In the following lemma we show that the above insertion procedure yields a coupling of the measures $(\Cycle_n)_{n\geq 3}$.

  \begin{lemma}\label{lem:coupling}
    Fix $q\geq 1$. For each $n\geq 3$, let $X^{(n)}=(X_1,\ldots,X_n)$ be random with law $\Cycle_{n,q}$. Let $I\in[n]$ be uniformly random and independent of $X^{(n)}$. Conditional on $I$ and $X^{(n)}$, let $Z$ be a uniformly random element of the set $[q]\setminus \{X_{I-1},X_I\}$ (where indices are interpreted modulo $n$). Finally, let $R\in[n+1]$ be uniformly random and independent of $I, X^{(n)}, $ and $Z$. Then
    $$
      \tau_R(X_1,\ldots,X_{I-1},Z,X_I,\ldots,X_n)\eqd X^{(n+1)}.
    $$
  \end{lemma}
  \begin{proof}
    We show that for all words $x\in[q]^{n+1}$,
    \begin{equation}\label{eq:insCyc}
      \mathbb P\bigl(\tau_R(X_1,\ldots,X_{I-1},Z,X_I,\ldots,X_n)=x\bigr)=\mathbb P\bigl(X^{(n+1)}=x\bigr).
    \end{equation}
    Indeed, the probability appearing on the left side of \eqref{eq:insCyc} can be written as
    $$
      \sum_{r=1}^{n+1}\mathbb P(R=r)\sum_{i=1}^n\mathbb P(I=i)\mathbb P\bigl((X_1,\ldots,X_{i-1},Z,X_i,\ldots,X_n)=\tau_{-r}x\bigr),
    $$
    which by the definitions of $R,I,Z,$ and $\Cycle_{n,q}$ in \eqref{eq:cmarg} equals
    \begin{equation}\label{eq:insCyc3}
      \frac{1}{n(n+1)}\sum_{r=1}^{n+1}\sum_{i=1}^n\frac{\mathbbm{1}[x\text{ is cyclically proper}]}{q-2}\frac{B^\circ(\hat{\tau_{-r}x}_i)}{Z^\circ(n,q)}.
    \end{equation}
    The word $\hat{\tau_{-r}x}_i$ is a cyclic shift of $\hat{x}_{i-r}$, where indices are interpreted modulo $n+1$. Thus by Lemma~\ref{lem:shift}, we have that \eqref{eq:insCyc3} equals
    \begin{equation}\label{eq:insCyc4}
      \frac{\mathbbm{1}[x\text{ is cyclically proper}]}{n(n+1)(q-2)Z^\circ(n,q)}\sum_{r=1}^{n+1}\sum_{i=1}^nB^\circ(\hat{x}_{i-r}).
    \end{equation}
    For each $j\in[n+1]$, there are $n$ pairs $(r,i)\in[n+1]\times[n]$ for which $i-r\equiv j\mod{(n+1)}$. Combining this observation with Lemma~\ref{lem:partition} yields that \eqref{eq:insCyc4} equals
    \begin{equation}\label{eq:insCyc5}
      \frac{\mathbbm{1}[x\text{ is cyclically proper}]}{Z^\circ(n+1,q)}\sum_{j=1}^{n+1}B^\circ(\hat x_j).
    \end{equation}
    By \eqref{eq:rec} and the definition of $\Cycle_{n+1,q}$, the right side of \eqref{eq:insCyc5} equals that of \eqref{eq:insCyc}.
  \end{proof}

  \begin{proof}[Proof of Theorem~\ref{thm:main}]
    We prove by induction on $n\geq 1$ that the coloring $(C_1,\ldots,C_{n+2})$ arising from the Eden model has law $\Cycle_{n+2,q}$. This is clear for $n=1$, when both colorings are uniformly random. Next suppose this holds for $n$. Recall that $(C_1,\ldots,C_{n+2})$ is constructed from an Eden model cluster $T_n$, a uniform proper $q$-coloring of its dual triangulation $D_n$, and an independent uniform vertex $u$ of $D_n$. One obtains a coupling with the corresponding objects at time $n+1$ as follows. Choose a uniform triangle in $\mathbb D$ sharing an edge with the outer boundary of $D_n$, append it to $D_n$ to form $D_{n+1}$, choose a uniformly random non-clashing color for the new vertex, and finally select an independent uniform vertex of $D_{n+1}$. By Lemma~\ref{lem:coupling} it follows that the coloring of the $(n+3)$-cycle obtained from these data has law $\Cycle_{n+3,q}$, completing the induction.

    \sloppypar Combining the result of the previous paragraph with \eqref{eq:cmarg}, \eqref{eq:vmarg} , and Lemma~\ref{lem:rest}, yields the restriction property. Applying \cite[Theorem 1]{HL} yields $k$-dependence.
  \end{proof}

\section{Markov chains}\label{sec:markov}
  We give simple conceptual proofs that the 1- and 2-color marginals of the 3- and 4-colorings on the cycle are as claimed. Before giving the details, we explain the structure of the proof. For fixed $q\in\{3,4\}$, let
  $$
  X^{(n)}=(X_1,\ldots,X_n)=\bigl(X_1^{(n)},\ldots,X_n^{(n)}\bigr)
  $$
  be a coloring with law $\Cycle_{n,q}$. Recall that we are interested in the indicator of the support of one or two of the colors, i.e. the binary vector $J^{(n)}=(J_1,\ldots,J_n)$ where
  $$
  J_i=J_i^{(n)}:=\mathbbm{1}[X_i^{(n)}\in A],
  $$
  and where $A$ is $\{1\}$ or $\{1,2\}$. Our goal is to show for each $n$ that $J^{(n)}$ is equal in law to a certain `target' binary vector $Q^{(n)}$ that is defined in terms of random permutations or binary strings.

  We will couple the random vectors $J^{(n)}$ for $n\geq 3$ by considering them as the states of a Markov chain. We will specify the Markov transition rule from $J^{(n)}$ to $J^{(n+1)}$. Then we will construct another Markov chain whose state at time $n$ is equal in law to the target random vector $Q^{(n)}$. Finally we will show that even though their descriptions differ, the two Markov transition rules coincide.

  In fact, we will couple the colorings $X^{(n)}$ themselves via a Markov chain indexed by $n$. The vector $J^{(n)}$ is a deterministic function, $F$ say, of $X^{(n)}$. In general, applying a function to the state of a Markov chain does not yield another Markov chain. However, this \textit{will} hold in our case. Specifically, it will follow because the conditional law of $F(X^{(n+1)})$ given $X^{(n)}=x$ will be identical for all $x$ having the same image under $F$. The same remarks will apply to the target laws also: the random vector $Q^{(n)}$ can be expressed as a deterministic function of a random permutation or a binary string. We will give a natural Markov chain on permutations or binary strings which yields a Markov chain for $Q^{(n)}$.

  All the Markov chains we consider will involve random insertions in the cycle. These insertions take the form described in the previous section. Namely, given $X^{(n)}=(X_1,\ldots,X_n)$, we will choose a uniformly random index $I\in\{1,\ldots,n\}$ (independently of $X^{(n)}$), insert a new element just before element $X_I$ (which can depend on $I$ and $X^{(n)}$), and apply a uniformly random rotation (independently of $I$ and $X^{(n)}$).

  We now proceed with the details of the proofs.

  \begin{proof}[Proof of Theorem~\ref{thm:onetwo}(i)]
    We show that the locations of colors 1 and 2 in the 4-coloring of the $n$-cycle have the same joint law as the descent set of a uniform permutation of $n$-cycle for all $n\geq 3$. This is easy to verify directly when $n=3$, in which case both random sets are uniformly distributed over subsets of size 1 or 2. To prove the result for other $n$, we show that the Markov transition rules are equivalent.

    Let $X^{(n)}$ be random and $\Cycle_{n,4}$-distributed. Define the function $F$ as follows:
    $$
      F(x_1,\ldots,x_n):=\bigl(\mathbbm{1}[x_i\in\{1,2\}]\bigr)_{i=1}^n.
    $$
    Set $J^{(n)}:=F(X^{(n)})$ to be the binary vector indicating the locations of colors 1 and 2.

    The Markov transition rule for $J^{(n)}$ obtained by applying the function $F$ to the transition rule in Lemma~\ref{lem:coupling} is the following. Conditional on $J^{(n)}$, sample $(I,B)$ uniformly from $[n]\times \{0,1\}$, let
    $$
      Z=\begin{cases}
        1-J_I,& J_{I-1}=J_I\\
        B,& J_{I-1}\not=J_I
      \end{cases},
    $$
    and take $J^{(n+1)}$ to be the image of $(J_1,\ldots,J_{I-1},Z,J_I,\ldots,J_n)$ under an independent uniformly random rotation. This transition rule is equivalent to choosing a random insertion location, querying the bit of one of the two neighbors at random, and inserting the complement of the queried bit.

    Let $\Pi^{(n)}=(\Pi_1,\ldots,\Pi_n)\in[n]^n$ be a uniformly random permutation. That is, $\Pi^{(n)}$ is distributed uniformly over the set of words in $[n]^n$ in which each symbol appears exactly once. Let $G$ be the function
    $$
      G(x_1,\ldots,x_n):=\bigl(\mathbbm{1}[x_i>x_{i+1}]\bigr)_{i=1}^n,
    $$
    and set $Q^{(n)}:=G(\Pi^{(n)})$.

    Before describing the transition rule for $Q^{(n)}$, we describe one for $\Pi^{(n)}$. Given $\Pi^{(n)}$, let $I\in [n]$ be an independent uniformly random index and let $\Pi^{(n+1)}$ be the result of applying a uniformly random rotation to one of the two tuples
    $$
      (\Pi_1,\ldots,\Pi_{I-1},n+1,\Pi_I,\ldots,\Pi_n)\text{ or }(\Pi_1+1,\ldots,\Pi_{I-1}+1,1,\Pi_I+1,\ldots,\Pi_n+1),
    $$
    chosen uniformly at random and independently of the rotation and $I$. It is easily seen that the law of $\Pi^{(n+1)}$ is uniform on the set of permutations.

    The Markov transition rule for $Q^{(n)}$ obtained by applying the function $G$ to the transition rule of $\Pi^{(n)}$ is the following. Conditional on $Q^{(n)}$, sample $(I,B)$ uniformly from $[n]\times \{0,1\}$ and take $Q^{(n+1)}$ to be the image of
    $(Q_1,\ldots,Q_{I-2},B,1-B,Q_I,\ldots,Q_n)$ under an independent uniformly random rotation. This transition rule is equivalent to choosing a random symbol and replacing it with one of the strings $01$ or $10$, chosen uniformly at random and independently of $Q^{(n)}$ and $I$.

    The transition rules for $J^{(n)}$ and $Q^{(n)}$ coincide, as both are equivalent to the following. Conditional on $J^{(n)}=Q^{(n)}=x$, let $U$ be chosen uniformly at random from the set
    $$
      \bigl\{i-\tfrac14\bigr\}_{i=1}^n\cup\bigl\{i+\tfrac14\bigr\}_{i=1}^n
    $$
    and let $V$ denote the integer closest to $U$. Insert the bit $1-x_V$ between positions $\lfloor U\rfloor$ and $\lceil U\rceil$ of $x$.

    Since $J^{(3)}$ and $Q^{(3)}$ have the same law and the Markov transition rules for $J^{(n)}$ and $Q^{(n)}$ agree for all $n\geq 3$, it follows that $J^{(n)}$ and $Q^{(n)}$ are equal in law.
  \end{proof}

  \begin{proof}[Proof of Theorem~\ref{thm:onetwo}(ii)]
    We show that for all $n\geq 3$, the locations of color 1 in the 3-coloring of the $n$-cycle have the same joint law as the peak set of a uniformly random permutation of the $n$-cycle. This is easy to verify directly when $n=3$, in which case both random sets are uniformly distributed over the singletons.

    Let $X^{(n)}$ be random and $\Cycle_{n,3}$-distributed, let $F$ be the function
    $$
      F(x_1,\ldots,x_n):=\bigl(\mathbbm{1}[x_i=1]\bigr)_{i=1}^n,
    $$
    and let $J^{(n)}:=F(X^{(n)})$. Take $\Pi^{(n)}=(\Pi_1,\ldots,\Pi_n)$ to be a uniformly random permutation and let $G$ be the function
    $$
      G(x_1,\ldots,x_n):=\bigl(\mathbbm{1}[x_{i-1}<x_i>x_{i+1}]\bigr)_{i=1}^n,
    $$
    with indices interpreted modulo $n$. Set $Q^{(n)}:=G(\Pi^{(n)})$.

    The Markov transition rule for $J^{(n)}$ obtained by applying the function $F$ to the transition rule in Lemma~\ref{lem:coupling} is the following. Conditional on $J^{(n)}$, sample $I$ uniformly from $[n]$, let
    $$
      Z=\begin{cases}
        1,& J_{I-1}=J_I=0\\
        0,& \text{otherwise}
      \end{cases},
    $$
    and take $J^{(n+1)}$ to be the image of $(J_1,\ldots,J_{I-1},Z,J_I,\ldots,J_n)$ under an independent uniformly random rotation.

    Before describing the transition rule for $Q^{(n)}$, we describe one for $\Pi^{(n)}$. Given $\Pi^{(n)}$, let $I\in[n]$ be an independent uniformly random index and let $\Pi^{(n+1)}$ be the result of applying a uniformly random rotation to
    $$
      (\Pi_1,\ldots,\Pi_{I-1},n+1,\Pi_I,\ldots,\Pi_n).
    $$
    It is easily seen that the law of $\Pi^{(n+1)}$ is uniform on the set of permutations.

    The Markov transition rule for $Q^{(n)}$ obtained by applying the function $G$ to the transition rule of $\Pi^{(n)}$ is the following. Conditional on $Q^{(n)}$, sample $I$ uniformly from $[n]$ and take $Q^{(n+1)}$ to be the image of $(Q_1,\ldots,Q_{I-2},0,1,0,Q_{I+1},\ldots,Q_n)$ under an independent uniformly random rotation.

    We show that the transition rules for $J^{(n)}$ and $Q^{(n)}$ coincide. The state space consists of binary strings that lack adjacent ones. We identify any such a string, $x$, with a partition of the set $\bigl\{i+\tfrac12\bigr\}_{i=1}^n$, by regarding the ones as endpoints of blocks. Both transition rules are equivalent to the following. Conditional on $J^{(n)}=Q^{(n)}=x$, let $U$ be uniformly random element of the set $\bigl\{i+\tfrac12\bigr\}_{i=1}^n$. If all of $U-1$, $U$, and $U+1$ belong to the same block of $x$, split the block at $U$ by inserting a one. Otherwise, grow the length of the block containing $U$ by inserting a zero.

    Since $J^{(3)}$ and $Q^{(3)}$ have the same law and the Markov transition rules for $J^{(n)}$ and $Q^{(n)}$ agree for all $n\geq 3$, it follows that $J^{(n)}$ and $Q^{(n)}$ are equal in law.
  \end{proof}

  We use the following term in the proof of Theorem~\ref{thm:onetwo}(iii). A \textbf{hard-core} process on a graph $G$ is a process $(J_v)_{v\in V}$ such that each $J_v$ takes values in $\{0,1\}$, and almost surely we do not have $J_u=J_v=1$ for adjacent vertices $u,v$.
  \begin{proof}[Proof of Theorem~\ref{thm:onetwo}(iii)]
    Both process are hard-core and 1-dependent. As shown in \cite{HL} (in the proof of Lemma~23), the law of a 1-dependent hard-core process on any graph is determined by its one-vertex marginals. Since this equals $\tfrac14$ for both processes, it follows that they are equal in law.
  \end{proof}
  We remark that Theorem~\ref{thm:onetwo}(iii) may also be proven via an argument involving Markov transition rules, similar to the proofs of Theorem~\ref{thm:onetwo}(i) and (ii), which we now sketch. The transition rule for the ones process inserts a random bit between two zeros, and otherwise inserts a zero. In terms of the half-integer blocks picture, you grow the block if the insertion location is at an endpoint, and otherwise you flip a coin to decide if to split or grow a block. The $\mathbbm{1}[B_i>B_{i+1}]$ process can also be interpreted in this way.

\section{Block factor}\label{sec:block}
  Recall that $\iota:{1,2,3,4}\to\{1,2,\Star\}$ is the function that fixes $1$ and
  $2$ but maps $3$ and $4$ to $\Star$.  Recall that a process
  $X=(X_i)_{i\in\Z}$ is said to be an $r$-block-factor of $U=(U_i)_{i\in\Z}$ if
  $X_i=f(U_{i},\ldots,U_{i+r-1})$ for all $i$ and a fixed function $f$.

  In this section we prove Theorem~\ref{onetwostar}.
  The proof uses the following lemma which characterizes
  $2$-block-factor hard-core processes.

  \begin{lemma}\label{rectangle}
    Suppose that a hard-core process $Z=(Z_i)_{i\in\Z}$ is a
    $2$-block-factor of an i.i.d.\ process $U$.  Then, almost
    surely
    $$Z_i=\mathbbm{1}\bigl[(U_i,U_{i+1})\in S\bigr],\quad i\in\Z,$$
    for some sets $S$ and $A$ satisfying
    $$S\subseteq A\times A^C.$$
     Moreover, we have $\E Z_0\leq \tfrac14$, with
    equality if and only if $\P(U_0\in A)=\tfrac12$ and
    $S=A\times A^C$ up to null sets.
  \end{lemma}

  Above and in the following proofs, we say that two sets are equal up to null
  sets if their symmetric difference is a null set.  Similarly, we write
  $E\subseteq F$ up to a null set if $E\setminus F$ is null.  Note that the
  inequality $\E Z_0\leq \tfrac14$ holds even for stationary $1$-dependent
  hard-core processes; see e.g.\ \cite[Corollary~17]{HL}.

  \begin{proof}[Proof of Lemma \ref{rectangle}]
    Without loss of generality, suppose that the $U_i$ are
    uniformly distributed on $[0,1]$.  By definition, any
    $\{0,1\}$-valued $2$-block-factor $Z$ is of the form
    $Z_i=\mathbbm{1}[(U_i,U_{i+1})\in S]$ for some $S\subseteq
    [0,1]^2$.  Our task is to show that $S$ has the claimed
    form. Define
    $$A:=\Bigl\{u_0\in[0,1]: \P\bigl[ (u_0,U_1)\in S \bigr]>0 \Bigr\}.$$

    First note that $S\subseteq A\times[0,1]$ up to a null set by Fubini's theorem.  (Indeed, if
    $\P(U_0\notin A)>0$ then $\P[(U_0,U_1)\in S\mid U_0\notin A]=0$; thus
    $\P[(U_0,U_1)\in S,\;  U_0\notin A]=0$.)

    We claim that also $S\subseteq [0,1]\times A^C$ up to a
    null set.  If not, then the event $E:=\{(U_0,U_1)\in S,\;
    U_1\in A\}$ has positive probability.  But $U_2$ is
    independent of $E$.  Therefore, using the definition of
    $A$, we have $\P[(U_1,U_2)\in S\mid E]>0$.  Hence,
    $$\P(Z_0=Z_1=1)\geq \P\bigl[E,\; (U_1,U_2)\in S\bigr]>0,$$
    which contradicts $Z$ being hard-core, thus proving the
    claim.

    The two inclusions proved above imply that $S\subseteq A\times A^C$ up to a
    null set.  Adjusting $S$ by a null set if necessary, the first assertion of
    the lemma follows.

    To prove the second assertion, note that
    $$\E Z_0=\P\bigl[(U_i,U_{i+1})\in S\bigr]\leq
    \P\bigl[(U_i,U_{i+1})\in A\times A^C\bigr]=a(1-a)\leq
    \tfrac14,$$ where $a:=\P(U_0\in A)$.  Equality throughout
    implies the stated conditions.
  \end{proof}

  \begin{proof}[Proof of Theorem~\ref{onetwostar}]
    Let $X$ be the $1$-dependent $4$-coloring of $\Z$, let $W_i=\iota(X_i)$, and
    suppose for a contradiction that $W$ is a $2$-block factor of the i.i.d.\
    process $U$. Without loss of generality, suppose that the $U_i$ are uniformly
    distributed on $[0,1]$.
     Let $Y_i=\mathbbm{1}[X_i=1]$ and $Z_i=\mathbbm{1}[X_i=2]$. Since $Y$ and $Z$ are functions
    of $W$, they are also $2$-block-factors of $U$.  By color symmetry of the 1-dependent 4-coloring (Theorem~\ref{thm:main} or \cite[Theorem~4(i)]{HL}) or by the inequality in Lemma~\ref{rectangle} we
    have $\E Y_0=\E Z_0=\tfrac14$, and so Lemma \ref{rectangle}
    implies that there exist sets $A,B\subset[0,1]$ each of
    measure $\tfrac12$ such that
    $$Y_i=\mathbbm{1}\bigl[(U_i,U_{i+1})\in A\times A^C\bigr]; \qquad
    Z_i=\mathbbm{1}\bigl[(U_i,U_{i+1})\in B\times B^C\bigr]$$ almost surely.

    We claim that $A=B^C$ up to null sets. Indeed, $A\cap B$
    has positive measure if and only if $A^C\cap B^C$ does. But
    in that case,
    $$\P(Y_0=Z_0=1)\geq \P(U_0\in A\cap B)\,\P(U_1\in A^C\cap B^C)>0,$$ which
    is impossible.

    Thus, up to null sets, $W_i=\Star$ if and only if
    $(U_i,U_{i+1})\in A^2\cup(A^C)^2$.  But this gives the
    wrong law for the process of $\Star$'s. For instance,
    \begin{align*}
      \P(W_0=W_1=\Star)&=\P\bigl[(U_0,U_1,U_2)\in A^3\cup(A^C)^3\bigr]=\tfrac14,\\
      \intertext{whereas Theorem~\ref{thm:onetwo}(iii) gives}
      \P(W_0=W_1=\Star)&=\P\bigl[X_0,X_1\in\{3,4\}\bigr]=\P(U_0<U_1<U_2)=\tfrac16.
      \quad\qedhere
    \end{align*}
  \end{proof}

\section{Open problems}
  \begin{enumerate}
    \item For $(k,q)\in\{(1,5),(2,4),(3,3)\}$ and for all larger $k$ and $q$, is there a color-symmetric $k$-dependent $q$-coloring of the $n$-cycle for all $n\geq 3$? In particular, is there an analogue of the colorings in \cite{hhlMalCol} on the cycle?
    \item Let $\iota\colon \{1,2,3,4\}\to\{1,2,\Star\}$ fix $1,2$ and map $3,4$ to $\Star$. Let $X=(X_i)_{i\in\Z}$ be the 1-dependent 4-coloring of \cite{HL}. Is the process $(\iota(X_i))_{i\in\Z}$ a block factor of an i.i.d.\ process?
  \end{enumerate}

\section*{Acknowledgements}\label{sec:ack}
  AL and TH were supported by internships at Microsoft Research while portions of this work were completed. TH was also supported by a Microsoft Research PhD fellowship.

  \bibliographystyle{habbrv}
  \bibliography{coloring}

\end{document}